\documentclass[12pt]{amsart}
\usepackage{amssymb,amsthm,tikz}

\usetikzlibrary{matrix,arrows}

\newtheorem{theorem}{Theorem}
\newtheorem{lemma}[theorem]{Lemma}

\theoremstyle{definition}

\def\F{\rm F}

\newcommand{\Z}{\mathrm Z}
\newcommand{\ZZ}{\mathbb Z}

\newcommand{\cH}{{\mathcal{H}}}

\newcommand{\AGL}{\mathop{\mathrm{AGL}}}
\newcommand{\soc}{\mathop{\mathrm{soc}}}
\newcommand{\GL}{\mathop{\mathrm{GL}}}

\newcommand{\SL}{\mathop{\mathrm{SL}}}
\newcommand{\PSL}{\mathop{\mathrm{PSL}}}

\newcommand{\Aut}{\mathop{\mathrm{Aut}}}

\newcommand{\barG}{{\bar{G}}}
\newcommand{\barGa}{{\bar{\Gamma}}}

\newcommand{\U}{\hskip 1pt \mathcal U}
\newcommand{\p}{\wp}

\renewcommand{\wr}{\mathop{\mathrm{wr}}}

\newcommand{\norml}{\vartriangleleft}

\begin{document}
\title{On the orders of arc-transitive graphs\footnote{This paper will appear in {\em Journal of Algebra}.}}

\author[M.D.E.\ Conder]{Marston D.E.\ Conder}
\address{Marston D.E.\  Conder, \newline  ${}$\hskip 20pt 
Department of Mathematics, University of Auckland,  
  \newline ${}$\hskip 20pt 
  Private Bag 92019, Auckland 1142, New Zealand}
\email{m.conder@auckland.ac.nz}                

\author[C.H.\ Li]{Cai Heng Li}
\address{Cai Heng Li,  \newline ${}$\hskip 20pt 
School of Mathematics and Statistics, University of Western Australia, 
  \newline ${}$\hskip 20pt 
  Crawley WA 6009, Australia} 
 \email{cai.heng.li@uwa.edu.au}

\author[P. Poto\v{c}nik]{Primo\v{z} Poto\v{c}nik}
\address{Primo\v{z} Poto\v{c}nik, \newline ${}$\hskip 20pt 
Faculty of Mathematics and Physics, University of Ljubljana, \newline ${}$\hskip 20pt  Jadranska 21, SI-1000 Ljubljana, Slovenia,
 \newline ${}$\hskip 30pt   
also affiliated with 
 \newline ${}$\hskip 20pt  
IAM, University of Primorska, Muzejski trg 2, SI-6000 Koper, Slovenia; {\sc and}
 \newline ${}$\hskip 20pt
 IMFM, Jadranska 19, SI-1000 Ljubljana, Slovenia.}
\email{primoz.potocnik@fmf.uni-lj.si}

 \dedicatory{Dedicated to the  memory of \'Akos Seress}

\begin{abstract} 
A graph is called {\em arc-transitive\/} (or {\em symmetric\/}) if its automorphism group 
has a single orbit on ordered pairs of adjacent vertices, and $2$-arc-transitive 
its automorphism group has a single orbit on ordered paths of length $2$. 
In this paper we consider the orders of such graphs, for given valency. 
We prove that for any given positive integer $k$, 
there exist only finitely many connected $3$-valent $2$-arc-transitive 
graphs whose order is $kp$ for some prime $p$, 
and that if $d\ge 4$, then there exist only finitely many connected $d$-valent $2$-arc-transitive 
graphs whose order is $kp$ or $kp^2$ for some prime $p$.
We also prove that there are infinitely many (even) values of $k$ for which 
there are only finitely many connected $3$-valent symmetric graphs of order $kp$ 
where $p$ is prime. 
\\[-30pt] 
\end{abstract}

\maketitle

\section{Introduction}
\label{sec:intro}

A graph is called {\em arc-transitive\/} (or {\em symmetric\/}) if its automorphism 
group has a single orbit on the set of all ordered pairs of adjacent vertices in the graph. 
The study of such graphs has a long and interesting history, highlighted at an early 
stage by ingenious work by Tutte~\cite{tutte,tutte2} on the cubic ($3$-valent) case. 

This paper concerns the orders of finite symmetric graphs of given valency.  

Vertex-transitive graphs of prime order were shown to be circulants 
(Cayley graphs for cyclic groups) by Turner~\cite{turner}, and then those which are 
symmetric were determined by Chao~\cite{Chao}.  
A few years later, Cheng and Oxley~\cite{ChengOxley} found all symmetric graphs 
of order $2p$ for $p$ prime.  (In fact Cheng and Oxley classified all graphs of order $2p$ 
that are both vertex- and edge-transitive, and proved that all of these graphs are symmetric.) 

More recently, numerous papers have been published in which the authors 
classify all symmetric graphs with given small valency (usually $3$, $4$ or $5$)
and with order of the form $kp$ or $kp^2$ for a fixed integer $k$ and variable prime $p$ 
(see \cite{FengKwak, FengKwak10p, FengKwakWang} for example), and we are 
aware of a number of other attempts to achieve such classifications, using voltage graphs  
and more general covering techniques. 
In many of these papers, the authors show that for a given $k$, there can be only finitely 
many such graphs.

We will show that this is always true when we restrict our attention to the case 
where the graph is $2$-arc-transitive (meaning that its automorphism group 
has a single orbit on the set of all ordered paths $(u,v,w)$ of length $2$) and 
has given valency greater than $3$, 
as well as in the case of $2$-arc-transitive cubic graphs of order $kp$.

${}$ 
\bigskip

In fact we will prove the following:

\begin{theorem}
\label{thm:f2} 
 Let $k$ and $d$ be given positive integers, with $d\ge 3$.
 If $d=3$, then there exist only finitely many connected $d$-valent $2$-arc-transitive 
 graphs of order $kp$ for some prime $p$.
 If $d\ge 4$, then there exist only finitely many connected $d$-valent $2$-arc-transitive 
 graphs of order $kp$ or $kp^2$ for some prime $p$.
 \end{theorem}
 
Note that this theorem fails for $3$-valent arc-transitive graphs of order $kp^2$, 
as shown by the existence of infinitely many $3$-valent $2$-arc-transitive graphs 
of order $6p^2$ (for $p$ prime), obtainable as $\ZZ_p^{\,2}$-covers of $K_{3,3}$; 
see \cite[Table 1]{solvable}, or \cite[Theorem 5.1]{ConMa}. 
It also fails for both $3$- and $4$-valent graphs of order $kp^3$, 
as exhibited by the existence of infinitely many $3$-valent $2$-arc-transitive graphs 
of order $4p^3$, obtainable as $\ZZ_p^{\, 3}$-covers of $K_4$ (see \cite{FengKwak} 
or \cite[Theorem 4.1]{ConMa}), 
and  infinitely many $4$-valent $2$-arc-transitive graphs of order $5p^3$, 
obtainable as $\ZZ_p^{\,3}$-covers of $K_5$ (see \cite[Table 1]{KuzmanK5}).

\smallskip 
Theorem~\ref{thm:f2} is a consequence of a much more general (but also more 
technical) theorem, the formulation of which requires some definitions.

\smallskip 
Throughout this paper, we will let $\Gamma$ be a finite connected simple undirected graph, 
and let $V(\Gamma)$, $E(\Gamma)$ and $A(\Gamma)$ be its vertex set, edge set 
and arc set respectively, 
where an {\em arc\/} is an ordered pair $(u,v)$ of adjacent vertices.
Similarly, if $s$ is a positive integer, then an {\em $s$-arc} of $\Gamma$ is an ordered $(s+1)$-tuple 
$(v_0,v_1,v_2,\dots,v_s)$ of vertices of $\Gamma$ in which any two consecutive vertices 
are adjacent and any three consecutive vertices are pairwise distinct,  
and we will denote by $A_s(\Gamma)$ the set of all $s$-arcs of $\Gamma$.

The group of all automorphisms of $\Gamma$ is denoted by $\Aut(\Gamma)$. 
If this group has a single orbit on $V(\Gamma)$, or on $E(\Gamma)$, or $A(\Gamma)$, 
or $A_s(\Gamma)$, then $\Gamma$ is said to be  {\em vertex-transitive}, {\em edge-transitive},
{\em arc-transitive}, or {\em $s$-arc-transitive},  respectively. 
The term {\em symmetric} is synonymous with {\em arc-transitive}, in this context. 
More generally, if a subgroup $G$ of $\Aut(\Gamma)$ acts transitively on $V(\Gamma)$,  
or $E(\Gamma)$, or $A(\Gamma)$, or $A_s(\Gamma)$), then we say that the graph $\Gamma$ 
is {\em $G$-vertex-transitive}, {\em $G$-edge-transitive},  {\em $G$-arc-transitive}, 
or {\em $(G,s)$-arc-transitive}, respectively.  

Next, for any subgroup $G$ of $\Aut(\Gamma)$, and for any vertex $v\in V(\Gamma)$, 
let $G_v$ be the stabiliser $\{g \in G \mid v^g = v\}$ of $v$ in $G$, 
and let $G_v^{\,\Gamma(v)}$ denote the permutation group induced by the 
action of $G_v$ on the neighbourhood $\Gamma(v)$ of $v$. 
Also denote the kernel of this action by $G_v^{\,[1]}$.
Note that  a $G$-vertex-transitive graph $\Gamma$ is $G$-arc-transitive 
if and only if  $G_v^{\,\Gamma(v)}$ is transitive on $\Gamma(v)$,
and is $(G,2)$-arc-transitive if and only if $G_v^{\,\Gamma(v)}$ is $2$-transitive on $\Gamma(v)$.

A permutation group for which the stabiliser of every point is trivial is called {\em semiregular}. 
A permutation group in which every non-trivial normal subgroup is transitive is called {\em quasiprimitive}.

Suppose from now on that $\Gamma$ is $G$-vertex-transitive.
If the group $G_v^{\,\Gamma(v)}$ is permutation isomorphic to some permutation group $L$
(for some and therefore every vertex $v$ of $\Gamma$), then we say that $G$ is {\em locally $L$}.
Similarly, if $G_v^{\,\Gamma(v)}$ is a quasiprimitive permutation group, then we say that $G$ is {\em locally quasiprimitive}.
Also following \cite{Verret}, we say that a transitive permutation group $L$ is {\em graph-restrictive} 
provided there exists a constant $c = c(L)$ such that whenever $G$ is an arc-transitive, locally $L$ group 
of automorphisms of a graph $\Gamma$, the order of the stabiliser $G_v$ is at most $c(L)$.

If $N$ is any subgroup of $\Aut(\Gamma)$, one may construct the quotient 
graph $\Gamma/N$, the vertex set of which is the set of $N$-orbits on $V(\Gamma)$, 
with two such orbits adjacent in $\Gamma/N$ whenever there is an edge between them in $\Gamma$.
Here we note that there is a natural graph epimorphism $\wp\colon \Gamma \to \Gamma/N$, 
mappng a vertex $v$ to the $N$-orbit $v^N$. 
If $\wp$ happens to map the neighbourhood $\Gamma(v)$ of every vertex $v\in V(\Gamma)$ 
bijectively onto the neighbourhood of $v^N$ in $\Gamma/N$,
then we say that $\p$ is a {\em $N$-regular covering projection} (or simply a {\em regular covering projection}).


\medskip

We can now state our more general theorem. 
Here we let $C_n$ and $D_n$ stand respectively for the cyclic group of order $n$ 
and the dihedral group of order $2n$, in their natural transitive actions on $n$ points. 

\begin{theorem}
\label{thm:fg}
Let $L$ be a quasiprimitive graph-restrictive permutation group of degree $d\ge 3$,  
with corresponding constant $c(L)$, and let $k$ be a given positive integer, 
and $p$ any prime satisfying $p\ge kc(L)$.
Also suppose there exists a $G$-arc-transitive graph $\Gamma$ of order $kp^\alpha$ 
where $\alpha = 1$ or $2$, such that $G_v^{\,\Gamma(v)}$ is permutation isomorphic to $L$, 
and let $P$ be a Sylow $p$-subgroup of $G$. Then the following hold{\,\em :}
\begin{itemize}
 \item[{\rm (a)}]
The subgroup $P$ is normal in $G$, has order $p^\alpha$, and acts semiregularly on $V(\Gamma)\,;$ 
\\[-10pt] 
 \item[{\rm (b)}]
If $k\ge 3$, $\bar{\Gamma} = \Gamma/P$ and $\bar{G} = G/P$, 
then the natural projection $\Gamma \to \bar{\Gamma}$ is a regular covering projection,
$\bar{\Gamma}$ has order $k$, and $\bar{G}$ is an arc-transitive 
and locally $L$ group of automorphisms of $\bar{\Gamma}\,;$ 
\\[-10pt] 
 \item[{\rm (c)}]
The stabiliser $G_v$ is isomorphic to a subgroup of $\Aut(P)\,;$ 
\\[-10pt] 
\item[{\rm (d)}]
 The degree $d$ of $L$ $($the valency of $\Gamma)$ is prime, and one of the following holds$\,:$ \\[-10pt] 
\begin{itemize}
\item[{\rm (i)\,}]
$P$ is cyclic, \ $p \equiv 1$ mod $2d$, \ $G_v^{\,\Gamma(v)}\cong C_d$, $|\bar{G}| = kd\,$ 
   and $\,\bar{G}/[\bar{G},\bar{G}] \cong C_{2d}\,;$ \\[-10pt] 
\item[{\rm (ii)}]
$\alpha =2$, $P$ is elementary abelian of order $p^2,$ and $G_v^{\,\Gamma(v)}\cong C_d$  or $D_d$.
\end{itemize}
\end{itemize}
\end{theorem}
%
%

Theorem~\ref{thm:fg} will be proved in Section~\ref{sec:proof-thmfg}, 
after some further background is given in Section~\ref{sec:further}, 
and Theorem~\ref{thm:f2} is proved in Section~\ref{sec:proof-thmf2}. 
Then in Section~\ref{sec:constructions} we describe a means for constructing 
examples of symmetric graphs of order $kp$ for a given positive integer $k$ 
and variable prime $p$, under certain conditions, 
and finally, we consider the special case of symmetric cubic graphs 
in Section~\ref{sec:3valentcase}. 

\smallskip
Before continuing, we comment on the assumptions made in Theorem~\ref{thm:fg} 
about quasiprimitivity and graph-restrictivness of the group $L$.

The importance of quasiprimitivity of the group $G_v^{\,\Gamma(v)}$ for a $G$-arc-transitive 
graph $\Gamma$ was first observed by Cheryl Praeger in \cite{Pqp}, 
after noticing that such pairs $(\Gamma,G)$ behave nicely with regard to 
taking a quotient $\Gamma/N$ by a normal subgroup $N$ of $G$. 
Local quasiprimitivity has now become a standard assumption in many applications 
of `quotienting' techniques. 

A classical topic in algebraic graph theory is the question whether the order of 
a vertex-stabiliser $G_v$ for a connected $d$-valent $G$-arc-transitive graph $\Gamma$ 
can be bounded by an absolute constant (depending only on $d$). 
A famous instance is the theorem of Tutte that gives $|G_v| \le 48$ when $d = 3$.
For larger $d$, the boundedness of $|G_v|$ depends not only on $d$, 
but also on the permutation group $G_v^{\,\Gamma(v)}$;  
for example, when $d=4$ the order of $G_v$ can be bounded by a constant 
provided that $G_v^{\,\Gamma(v)} \cong \ZZ_2^2$, $\ZZ_4$, $A_4$ or $S_4$. To capture this
phenomenon, the term  {\em graph-restrictiveness} was coined by Gabriel Verret in \cite{Verret}.

Using Verret's terminology, one can easily express several classical results and 
conjectures in a different way. 
For example, Tutte's theorem says that the two transitive groups of degree $3$, 
namely $C_3$ and $S_3$, are graph restrictive, with corresponding constants
$c(C_3) = 3$ and $c(S_3) = 48$. 
Similarly, it can be deduced from the work of Gardiner~\cite{Gard} that the alternating 
group $A_4$ and the symmetric group $S_4$ (both of degree $4$) are
graph restrictive, with corresponding constants $c(A_4) = 36$ and $c(S_4)=2^{4\,}3^{6}$. 

An even stronger theorem holds, thanks to work by Richard Weiss and Vladimir Trofimov, 
namely that every doubly transitive group is graph-restrictive. 
The proof of this fact can be found by putting together pieces from many papers, 
but a nice summary is given in the introduction to a later paper by Weiss \cite{dt}.

Also when this is taken together with another theorem proved in \cite{weissp}, 
it implies that every transitive permutation group of prime degree is graph-restrictive.
Other examples of graph-restrictive groups can be found in \cite{Verret,Verret2}, 
and a summary of all known graph-restrictive groups is given in \cite{PSV}. 
In particular, it is shown in \cite{PSV} that if $L$ is any transitive permutation group 
of degree at most~$8$, then $L$ is graph-restrictive if and only if  every normal subgroup 
of $L$ is either transitive or semiregular.

We conclude this discussion of graph-restrictiveness by pointing out two related 
conjectures. 
The first is the `Weiss conjecture', made by Richard Weiss \cite[Conjecture 3.12]{weissconj}; 
this can be re-worded to say that every primitive permutation group is graph-restrictive. 
The second is due to Cheryl Praeger \cite{PConj}, and essentially states that every 
quasiprimitive permutation group is graph-restrictive. 
Note that in view of the Praeger conjecture, the condition in Theorem~\ref{thm:fg}  
on graph-restrictiveness might very well not be needed, since it would 
follow automatically from quasiprimitivity.

 \section{Further background} 
 \label{sec:further}

We begin this section with a classical property of quotients of locally quasiprimitive graphs,
which will be used frequently in the proofs of our theorems.

\begin{lemma}
\label{lem:Cheryl}
\cite[Section 1]{Pqp}.
Let $\Gamma$ be a connected $G$-arc-transitive graph, let $N$ be a normal subgroup of  $G$,
and take $\bar{G}=G/N$ and $\bar{\Gamma}=\Gamma/N$, and $\bar{v} = v^N$ for each vertex $v$ of $\Gamma$. 
Then there is a natural $($but not necessarily faithful\,$)$ action of $\bar{G}$ on $\bar{\Gamma}$ 
as an arc-transitive group of automorphisms.
If also $G_v^{\,\Gamma(v)}$ is quasiprimitive and $N$ has at least $3$ orbits on $V(\Gamma)$,
then $N$ is semiregular on $\Gamma$, and the action of $\bar{G}$ on $\bar{\Gamma}$ is faithful.
Moreover, if $G$ is locally primitive, then the natural projection $\Gamma \to \bar{\Gamma}$ is a regular covering projection, and the groups $G_v^{\,\Gamma(v)}$ and $\bar{G}_{\bar{v}}^{\,\bar{\Gamma}(\bar{v})}$ 
are permutation isomorphic.
\end{lemma}

Lemma~\ref{lem:Cheryl} has the following easy consequence:

\begin{lemma}
\label{lem:qpquo}
Let $\Gamma$ be a connected $G$-arc-transitive graph, and let $N$ be a normal subgroup 
of $G$. If $G_v^{\,\Gamma(v)}$ is quasiprimitive and $N_v$ is non-trivial, 
then $N$ has at most two orbits on $V(\Gamma)$. 
\end{lemma}

We now give some other background theory that will be useful. 

\begin{lemma}
\label{lem:locallyC}
Suppose $\Gamma$ is a connected $G$-arc-transitive graph of valency $d$, and  
$G_v^{\,\Gamma(v)}$ is a cyclic group of order $d$. Then $G$ acts regularly on the arcs of
$\,\Gamma,$ and is generated by two elements of orders $d$ and $2$, such that the 
element of order $d$ generates $G_v$.
\end{lemma}

\begin{proof}
First, since $G_v^{\,\Gamma(v)}$ acts regularly on $\Gamma(v)$ for all $v$, the connectivity 
of $\Gamma$ implies that the kernel of the action of $G_v$ on $\Gamma(v)$ is trivial, 
and so $G_v \cong G_v^{\,\Gamma(v)}$. 
In particular, $G$ acts regularly on the arcs of $\Gamma$.
Next, by arc-transitivity, $G$ is generated by $G_v$ and an element $\tau$ 
that interchanges $v$ with one of its neighbours, say $w$.
Then $\tau^2$ stabilises the arc $(v,w)$, so $\tau^2=1$. 
Hence $G$ is generated by an element of order $d$ (generating $G_v$) 
and this element $\tau$ of order $2$.
\end{proof}

\smallskip

 \begin{lemma}
 \label{lem:bur}
Let $G$ be a quasiprimitive permutation group. 
Then $G$ contains at most two minimal normal subgroups, 
and its socle $M=\soc(G)$ is a direct product of pairwise isomorphic simple groups.
Furthermore, if $G$ is soluble, then $G$ is primitive of affine type$\,;$  in other words, 
$M$ is the only minimal normal subgroup of $G$, and is isomorphic to
an elementary abelian group $\ZZ_p^d$, and then $G$ is permutation isomorphic
 to a subgroup of the affine group $\AGL(d,p)\cong \ZZ_p^d \rtimes \GL(p,d)$
  in its natural action on the vector space $\ZZ_p^d$.
 \end{lemma}
 
 \begin{proof}
 This follows directly from the  first three paragraphs of  \cite[Section 3]{Cheryl}.
 \end{proof}
 
For a prime $p$ and a group $G$, let $O_p(G)$ denote the largest normal $p$-subgroup of $G$. 
Note that $O_p(H) \norml O_p(G)$ whenever $H\norml G$.
 
 \begin{lemma}
 \label{lem:Op}
 Let $\Gamma$ be a connected $G$-arc-transitive graph, 
 let $uv$ be an arc of $\Gamma$, and let $p$ be any prime.
 If $G_v$ contains a non-trivial  $p$-group $C$ as a normal subgroup,
 then either $C^{\,\Gamma(v)} \not = 1$  or $O_p(G_{uv}^{\,\Gamma(v)}) \not = 1$.
 \end{lemma}
 
 \begin{proof}
 Let $C$ be any non-trivial normal $p$-group of $G_v$, and suppose that $C^{\,\Gamma(v)} = 1$. 
 Then $C\le G_v^{[1]} \le G_{uv}$,
 and since $C$ is normal in $G_v$, it is follows that $C$ is normal both in $G_{uv}$ and $G_v^{[1]}$. In particular,
 $O_p(G_v^{[1]})$ and $O_p(G_{uv})$ are non-trivial. 
 Then since $G_v^{[1]} \norml G_{uv}$, it follows that $ O_p(G_v^{[1]}) \norml O_p(G_{uv})$.
 
Now suppose also that $O_p(G_{uv}^{\,\Gamma(v)}) = 1$. 
Then $O_p(G_{uv})$ is a (normal) subgroup of $G_v^{[1]}$, 
and since $ O_p(G_v^{[1]}) \norml O_p(G_{uv})$, it follows that $O_p(G_{uv}) = O_p(G_v^{[1]})$.
Hence $O_p(G_{uv})$ is a characteristic subgroup of $G_{uv}$ as well as one of $G_v^{[1]}$. 
Then since $G_{uv}$ is normal (of index $2$) in the
edge-stabiliser $G_{\{u,v\}}$ and $G_v^{[1]}$ is normal in $G_v$, 
we find that $O_p(G_{uv})$ is normal in $\langle G_v, G_{\{u,v\}} \rangle = G$,  
with the latter equality following from connectedness of $G$. 
But then $O_p(G_{uv})$ acts trivially on the arcs of $\Gamma$, which contradicts
the fact that $O_p(G_{uv}) \not = 1$. 
This shows that $O_p(G_{uv}^{\,\Gamma(v)}) \not= 1$, as required.
 \end{proof}

 \smallskip
 
\begin{lemma} 
\label{lem:leash}
Let $\Gamma$ be a finite connected $G$-vertex-transitive graph. 
Then every simple section of $G_v$ is also a section of $G_v^{\,\Gamma(v)}$.
In particular, if $G_v^{\,\Gamma(v)}$ is soluble, then so is $G_v$, 
and if the prime $p$ does divides $|G_v|$, then also $p$ divides $|G_v^{\,\Gamma(v)}|$. 
 \end{lemma}
 
\begin{proof}
This is almost folklore, and a more general version can be found in \cite{leash}.
\end{proof}

In the proof of our main theorem, we will need the following fact about the subgroups of $\GL(2,p)$.
We thank Pablo Spiga for offering us a proof, which uses the fact 
that when the prime $p$ is congruent $\pm 1$ mod $5$, the group $\SL(2,p)$ 
contains subgroups isomorphic to $\SL(2,5)$, and all such subgroups are maximal 
(see \cite[p.\,417, Ex.\,7]{suzuki}).

\begin{lemma}
\label{lem:charT}
Let $p$ be a prime congruent $\pm 1$ modulo $5$,
let $G=\GL(2,p)$, 
and let $T$ be the subgroup of order $2$ in $G$ generated by the negative identity matrix $-I_2$.
Also let $H$ be a subgroup of $G$ isomorphic to $\SL(2,5)$, 
and let $N$ be the normaliser of $H$ in $G$.
Then $T$ is a characteristic subgroup of $N$, 
and of every subgroup of $N$ containing $T$. 
\end{lemma}

\begin{proof}
Let $Z$ be the centre of $G$. We will show first that $N = ZH$.
Let $S = SL(2,p)$, and let $\cH$ be the set of all subgroups of $G$ isomorphic to $\SL(2,5)$. 
Now consider the action of $G$ on $\cH$ by conjugation. 
Then $G$ is transitive on $\cH$, while $S$ has two orbits on $\cH$, 
of equal size (since $S$ is normal in $G$); see for example \cite[p.\,416, Ex.\,2]{suzuki}.
The stabiliser of $H$ in $G$ is $N_G(H) = N$, while the stabiliser of $H$ is $S$ 
is $N_S(H) = H$, since $H$ is maximal but not normal in $S$. 
Hence by the orbit-stabiliser theorem, we have $|G| = |\cH||N|$ and $|S| = |\cH||H|/2$, 
and so $|Z| = p-1 = |G/S| = |G|/|S| = 2|N|/|H|$. 
This implies $|ZH| = |Z| |H| / |Z\cap H| = |Z| |H| /2 = |N|$, 
and then since $N$ contains both $Z$ and $H$, it follows that $N=ZH$.

Next, $H$ has just one involution, and this must be the unique involution 
in $S$, namely $-I_2$, which also lies in the central subgroup $Z$ of $ZH = N$, 
and hence must be the only involution in $N$. 
It follows that  the subgroup $T$ generated by this involution is
characteristic in $N$, and also in any subgroup of $ZH = N$ containing $T$.
\end{proof}

We will also use the following. 

\begin{lemma}
\label{lem:HK}
Let  $H$ and $K$ be normal subgroups of $G$, the orders of which are coprime. 
Then $HK\cap S=(H\cap S)(K\cap S)$ for every subgroup $S$ of $G$.
\end{lemma}

\begin{proof}
Clearly $(H\cap S)(K\cap S) \subseteq HK\cap S$, 
so it suffices to prove the reverse inclusion.
Now suppose $x\in H$ and $y\in K$, with $xy\in S$.
If $\alpha$ and $\beta$ are the orders of $x$ and $y$, respectively, 
then $\gcd(\alpha,\beta) = 1$ and it follows that there exists an integer $\gamma$ 
with $\alpha\gamma  \equiv 1$ mod $\beta$. 
Also because $H$ and $K$ have coprime orders, they intersect trivially, 
and so $[H,K] \subseteq H \cap K = 1$, which means that $H$ and $K$ centralise each other. 
In particular, $x$ commutes with $y$, and therefore
$y^\alpha = x^\alpha y^\alpha = (xy)^\alpha \in S$, 
which in turn gives $y = y^{\alpha\gamma} = (y^\alpha)^\gamma \in S$, and then 
also $x = (xy)y^{-1} \in S$. 
Thus $x\in H\cap S$ and $y \in K\cap S$, and so $xy\in (H\cap S)(K\cap S)$, as required.
\end{proof}

 \section{Proof of Theorem~\ref{thm:fg}}
 \label{sec:proof-thmfg}

We prove Theorem~\ref{thm:fg} in stages, beginning with the following: 

\begin{lemma}
\label{lem:GvGL}
Let $\Gamma$ be a finite connected $G$-arc-transitive graph of valency $d \ge 3$, 
such that  $G_v^{\,\Gamma(v)}$ is quasiprimitive. 
Also suppose that $G$ has an abelian normal Sylow $p$-subgroup $P$ acting semiregularly on the vertex-set of $\Gamma$,  
where $p$ is an odd prime.
Then the centraliser $C_G(P)$ of $P$ in $G$ is a direct product 
$J \times P$ for some normal subgroup $J\hskip -1pt$ of $\hskip 1pt G$, 
and $C_G(P)$ acts semiregularly on $V(\Gamma)$. 
In particular, $\Aut(P)$ contains a subgroup isomorphic to $G_v$.
Moreover, if the valency $d$ of $\Gamma$ is prime, 
and $G_v^{\,\Gamma(v)}$ is cyclic $($of order $d)$    
and also $P$ is cyclic, 
then $p \equiv 1$ {\em mod} $2d$, 
each of the groups $G$ and $\barG = G/P$ is generated by two elements of orders $d$ and $2$, 
and $[G,G]=C_G(P)$, $\,[\barG,\barG] \cong J$, and $G/[G,G] \cong \barG / [\barG, \barG] \cong C_{2d}$.
\end{lemma}

\begin{proof} 
First note that since $P$
acts semiregularly on $V(\Gamma)$, $p$ divides $|V(\Gamma)|$. 

Now let $C = C_G(P)$.  Then $C$ contains $P$ and is normal in $G$ (with factor group $G/C$ 
isomorphic to a subgroup of $\Aut(P)$). 
By the Schur-Zassenhaus theorem, we find that $P$ has a complement in $C$, 
and therefore $C = J \times P$ for some subgroup $J$ of $C$.  
Note that $J$ is a Hall $p'$-subgroup of $C$, so $J$ is characteristic in $C$ 
and hence normal in $G$.
Also $C_v = C \cap G_v = JP \cap G_v$, and so by Lemma~\ref{lem:HK}, 
we find that $C_v  = (J\cap G_v)(P\cap G_v) = J_{v}P_{v} = J_v$ (since $P_v = 1$). 

If $C_v$ (and therefore $J_v$) is non-trivial, then by Lemma~\ref{lem:qpquo} 
we know that $J$ has at most two orbits on $V(\Gamma)$. 
But $J \norml G$, and $G$ is transitive on $V(\Gamma)$, so all the orbits of $J$ 
have the same length, and since $|J|$ is coprime to $p$, it follows that 
the number of orbits of $J$ on $V(\Gamma)$ is divisible by $p$, a contradiction. 
Thus $C_v = 1$, or in other words, $C$ is semiregular on $V(\Gamma)$. 
Furthermore, this implies that the image of $G_v$ in the factor group $G/C$ 
is $G_vC/C \cong G_v/(G_v \cap C) = G_v/C_v \cong G_v$, and so $G_v$ is 
isomorphic to a subgroup of $\Aut(P)$.  
This proves the first part of the lemma.

Next, we suppose that the valency $d$ is prime, and that 
$G_v^{\,\Gamma(v)}\cong C_d$, and $P$ is cyclic, say of order $p^\alpha$. 
By Lemma~\ref{lem:locallyC} we know that $G$ is generated by 
an element $h$ of order $d$ and an element $a$ of order $2$, 
and it follows that the abelianisation $G/[G,G]$ is a quotient of the 
group $C_d \times C_2 \cong C_{2d}$.
On the other hand, since $G/C$ is isomorphic to a subgroup 
of $\Aut(P) \cong \Aut(C_{p^\alpha})$, which is abelian, 
we find that $[G,G] \le C$, and hence also $G/C$ is a quotient of $C_{2d}$. 
We will show that $[G,G]=C$.

Recall that $C=J\times P$, with $J \norml G$, and consider the quotient $G/J$. 
We have $|G/J| = |P||G/C| = p^\alpha |G/C|$, and since $|G/C|$ divides $2d$, 
it follows that $|G/J|$ divides $2dp^\alpha$.  
Next, because $P$ is semiregular, we know that $p$ is coprime to the order of $G_v$, 
and in particular $d \ne p$.
Also $G/J$ is generated by $Jh$ and $Ja$, the orders of which divide the 
primes $d$ and $2$ respectively.   Now if one of these cosets were trivial, 
then $G/J$ would be cyclic of order $1,2$ or $d$, but then its order would be 
coprime to $p$, a contradiction.  
Hence $Jh$ and $Ja$ have orders $d$ and $2$, so $|G/J|$ is divisible by $2d$, 
and therefore by $2dp^\alpha$ (again since $|J|$ is coprime to $p$). 
Thus $|G/J| = 2dp^\alpha$.

In turn this implies that $|G/C|=|G/J|/p^\alpha = 2d$, and then since $[G,G] \le C$ 
and $G/[G,G]$ is a quotient of $C_{2d}$, we deduce that $[G,G]=C$ 
and $G/[G,G] = G/C \cong C_{2d}$. 

Moreover, since $G/C$ is isomorphic to a subgroup of $\Aut(P) \cong \Aut(C_{p^\alpha})$, 
which is cyclic of order $\phi(p^\alpha) = p^{\alpha-1}(p-1)$, 
we find that $2d = |G/C|$ divides $p^{\alpha-1}(p-1)$ and hence divides $p-1$.
Thus $p \equiv 1$ mod $2d$.

Finally, we consider the quotient $\barG = G/P$. 
This is generated by the images $Ph$ and $Pa$, 
which have orders $d$ and $2$ (since the latter are both coprime to $|P|$). 
Also $[\barG, \barG] = [G/P, G/P] \cong [G,G]P/P = CP/P = C/P \cong J$, 
and it immediately follows that  
$\barG/[\barG,\barG] \cong (G/P) / (C/P) \cong G/C \cong G/[G,G] \cong C_{2d}$. 
\end{proof}
\vskip 2pt

\begin{lemma}
 \label{lem:main}
 Let $\Gamma$ be a finite connected $G$-arc-transitive graph 
 with valency $d \ge 3$, and suppose $G_v^{\,\Gamma(v)}$ is quasiprimitive.  
Also suppose that the order of $\Gamma$ is $kp$ or $kp^2$ for 
some positive integer $k$ and some prime $p$ which divides neither $k$ 
nor $|G_v^{\,\Gamma(v)}|$, and that $G$ has a normal Sylow $p$-subgroup $P$.
Then parts {\em (a)} and {\em (c)} of Theorem~{\em \ref{thm:fg}} hold, 
and if the valency $d$ is prime, then so does part {\em (b)}. 
\end{lemma}
 
\begin{proof}
Suppose $|V(\Gamma)| = kp^\alpha$, so that $\alpha = 1$ or $2$. 
Because $p$ is coprime to $|G_v^{\,\Gamma(v)}|$, we know from Lemma~\ref{lem:leash} 
that $p$ is coprime to $|G_v|$, and it follows that $P_v=1$, 
or in other words, $P$ is semiregular on $V(\Gamma)$. 
Moreover, since $|G| = kp^\alpha |G_v|$ and $p$ is coprime to both $k$ and $|G_v|$, 
we find that $|P|=p^\alpha$.
In particular, $P$ is abelian (since $\alpha \le 2$), so the conditions of Lemma~\ref{lem:GvGL} 
are fulfilled, and therefore $G_v$ is isomorphic to a subgroup of $\Aut(P)$.
These observations prove parts (a) and (c). 

Next suppose $k\ge 3$.   Then the order of the quotient graph $\bar{\Gamma} = \Gamma/P$ is 
equal to the number of $P$-orbits on $V(\Gamma)$, which is $|V(\Gamma)| / |P| = k$.
Also since $k\ge 3$ and $G_v^{\,\Gamma(v)}$ is quasiprimitive,  an application 
of Lemma~\ref{lem:Cheryl} (with $N = P$) shows that $\bar{G}$ acts faithfully 
and arc-transitively on $\bar{\Gamma}$. 
Moreover, if the valency $d$ is prime then $G$ is locally primitive, and 
it follows (by Lemma~\ref{lem:Cheryl}) that the natural projection $\Gamma \to \bar{\Gamma}$ is 
a regular covering projection, and $G_v^{\,\Gamma(v)}$ and $\bar{G}_{\bar{v}}^{\,\bar{\Gamma}(\bar{v})}$ 
are permutation isomorphic, so that $\bar{G}$ is locally $G_v^{\,\Gamma(v)}$.
This proves part (b). 
\end{proof}

Now to complete the proof of Theorem~\ref{thm:fg}, all we need to do is prove part (d), 
which includes showing that the valency $d$ is prime.  
We will use the fact that $|P| = p$ or $p^2$, from which it follows that $P$ is cyclic 
or elementary abelian of rank $2$. 

We proceed by dealing with two special cases. 

\smallskip
 
\begin{lemma}
\label{lem:abGv}
If $G_v^{\,\Gamma(v)}$ is abelian, then part {\em (d)} of Theorem~{\em\ref{thm:fg}} holds.
\end{lemma}

\begin{proof} 
Since $G_v^{\,\Gamma(v)}$ is abelian, it is regular on $\Gamma(v)$, 
and since it is also quasiprimitive, it must be cyclic of prime order. 
Thus $d$ is prime, and $G_v^{\,\Gamma(v)}\cong C_d$. 
Moreover, since $G_v^{\,\Gamma(v)}$ is regular, it follows from connectivity of $\Gamma$ 
that the kernel $G_v^{[1]}$ is trivial, and hence $G_v \cong G_v^{\,\Gamma(v)}$. 
Finally, if $P$ is cyclic, then part (d)(i) of of Theorem~\ref{thm:fg} holds by Lemma~\ref{lem:main}; 
and on the other hand, if $P$ is not cyclic, then $\alpha = 2$ and $P$ is elementary abelian,
so part (d)(ii)  holds. 
\end{proof} 
\smallskip

\begin{lemma}
\label{lem:cyclicP}
If the Sylow subgroup $P$ of $G$ is cyclic, then part {\em (d)} of Theorem~{\em\ref{thm:fg}} holds.
\end{lemma}

\begin{proof} 
Since $P$ is cyclic of prime power order, $\Aut(P)$ is cyclic, and hence $G_v$ is cyclic, by part (c). 
In turn this implies that $G_v^{\,\Gamma(v)}$ is cyclic, and Lemma~\ref{lem:abGv} applies.
\end{proof} 

 \smallskip
For the rest of the proof we may now assume that $\alpha =2$, and $P$ is elementary abelian
 of order $p^2$. Also it suffices to show that $d$ is prime, and $G_v^{\,\Gamma(v)}\cong C_d$ or $D_d$.
The latter holds automatically when $d = 3$, so we can also assume that $d\ge 4$.
 \smallskip
 
Now since $P\cong \ZZ_p^2$, we know $\Aut(P)$ is isomorphic to $\GL(2,p)$, 
and by part (c) it follows that $G_v$ is isomorphic to a subgroup of $\GL(2,p)$.
 In view of this, we can think of $G_v$ as a subgroup of $\GL(2,p)$.
We will consider the intersection  
$$ H = G_v \cap \SL(2,p).$$
This is a normal subgroup of $G_v$, with cyclic quotient, since  
$$
 G_v / H \, = \, G_v/(G_v \cap \SL(2,p)) \, \cong \, G_v \SL(2,p) / \SL(2,p) \, \le \, \GL(2,p)/\SL(2,p) \, \cong \, C_{p-1}.
$$
Also if $H$ is contained in $G_v^{[1]}$, 
then $G_v^{\,\Gamma(v)} \cong G_v/G_v^{[1]} \cong (G_v/ H)/(G_v^{[1]}/H)$, 
which is a quotient of $G_v/ H$ and therefore cyclic, 
and then part (d) follows from Lemma~\ref{lem:abGv}.

We may therefore assume that $H$ is not contained in $G_v^{[1]}$, and hence that
$H^{\,\Gamma(v)}$ is a non-trivial normal subgroup of $G_v^{\,\Gamma(v)}$. 
Moreover, since $G_v^{\,\Gamma(v)}$ is quasiprimitive,
this implies that  $H^{\,\Gamma(v)}$ is transitive on $\Gamma(v)$.

\smallskip
On the other hand, $H$ is a subgroup of $\SL(2,p)$, with order coprime to $p$, 
and so it follows from the classification of subgroups of $2$-dimensional special linear groups 
\cite[Theorem 6.17]{suzuki} that $H$ is isomorphic to one of the following:
\begin{itemize}
\item[{\rm (1)}]
a cyclic group; \vskip 3pt  
\item[{\rm (2)}]
a metacyclic group $\langle\, x,y \ | \ x^n = y^2, \, y^{-1}xy=x^{-1} \, \rangle$ of order $2n\hskip 1pt$; \vskip 2pt  
\item[{\rm (3)}]
the group $\hat{S}_4$ of order $48$ with a unique involution $\tau$ such that $\langle \tau \rangle$ is
the centre of $\hat{S}_4$, and $\hat{S}_4/\langle \tau \rangle \cong S_4$; \vskip 3pt 
\item[{\rm (4)}]
the special linear group $\SL(2,3)$; \vskip 3pt 
\item[{\rm (5)}]
the special linear group $\SL(2,5)$, in cases where $p \equiv \pm 1$ mod $5$.
\end{itemize}
\medskip

\noindent
This  allows us to prove the following:
 
\begin{lemma}
\label{lem:dodd}
The valency $d$ is at least $5$.
Moreover, the group $H^{\,\Gamma(v)} \cong H/(H \cap G_v^{[1]})$ is either cyclic or dihedral, 
or isomorphic to a quotient of $A_4$, $S_4$ or $A_5$.
\end{lemma}
  
 \begin{proof}
 Suppose first that $H$ has odd order.  Then $H$ is cyclic (since all the groups in 
 cases (2) to (5) above have even order), and therefore $|H^{\,\Gamma(v)}|$ is cyclic 
 of odd order.  By transitivity of $H^{\,\Gamma(v)}$ on $\Gamma(v)$, it follows that 
 the valency $d = |\Gamma(v)|$ is odd. Since $d\ge 4$, it follows that $d\ge 5$. 
 
Hence we may assume that $H$ has even order.  In that case, $H$ must contain the centre $T$ 
of $\SL(2,p)$, generated by the unique involution $-I_2\,$ in $\SL(2,p)$. 
In particular, the latter is the only involution in $H$, so $T$ is characteristic in $H$, 
and hence normal in $G_v$. 
On the other hand, $G_v^{\,\Gamma(v)}$ is quasiprimitive, and so contains no 
non-trivial normal subgroups of order less than $d = |\Gamma(v)|$. 
It follows that $T^{\,\Gamma(v)}  =1$, or equivalently, $T\le G_v^{[1]}$.
Then by Lemma~\ref{lem:Op}, we find that  $O_2(G_{uv}^{\,\Gamma(v)}) \not = 1$. 
If $d=4$, then $G_v^{\Gamma(v)}$ is permutation isomorphic to $A_4$ or $S_4$ in their natural actions on $4$ points (these being
the only quasiprimitive permutation groups of degree $4$).
Then $G_{uv}^{\,\Gamma(v)} \cong C_3$ or $S_3$. However, neither of these two groups contains a non-trivial normal $2$-subgroup.
This contradiction shows that $d\ge 5$.

Moreover, since $T\le H \cap G_v^{[1]}$, we know that 
 \begin{equation*}
  H^{\,\Gamma(v)}  \cong H/(H \cap G_v^{[1]}) \cong (H/T) / ((H \cap G_v^{[1]}) / T), 
 \end{equation*}
and so $H^{\,\Gamma(v)}$ is a quotient of $H/T$.
By inspection of the groups in cases (1) to (5), we see that $H/T$ is 
cyclic in case (1), dihedral of order $2n$ in case (2), $\,S_4$ in case (3),
$\,\PSL(2,3) \cong A_4$ in case (4), and $\,\PSL(2,5) \cong A_5$ in case (5), 
and the rest follows. 
\end{proof}

We complete the proof of Theorem~\ref{thm:fg} by considering the 
cases (1) to (5) above in more detail. 
 
\begin{lemma}
\label{lem:cases1&2}
Part {\em (d)} of Theorem~{\em\ref{thm:fg}} holds in cases {\em (1)} and {\em (2)}. 
\end{lemma}
  
 \begin{proof}
In these two cases, $H$ is soluble, and then since $G_v/H = G_v/(G_v \cap \SL(2,p))$ is cyclic, 
both $G_v$ and $G_v^{\,\Gamma(v)}$ are soluble too. 
By Lemma~\ref{lem:bur}, we find that the socle $M = \soc(G_v^{\,\Gamma(v)})$ is an 
elementary abelian group, of order $q$, say, and that $G_v^{\,\Gamma(v)}$ is permutation 
isomorphic to a subgroup of $\AGL(1,q)$, so $d = q$, and $M$ is the  only minimal normal subgroup 
of $G_v^{\,\Gamma(v)}$.  
Hence in particular,  $M = \soc(G_v^{\,\Gamma(v)})$ is a subgroup of $H^{\,\Gamma(v)}$.
But $H^{\,\Gamma(v)}$ is a quotient of $H/T$, and so is either cyclic or dihedral, 
and therefore $M$ is cyclic or dihedral, which implies that $M \cong C_q$ or $C_2\times C_2$. 
In the latter case, however, $d = q = 4$, which is impossible by Lemma~\ref{lem:dodd}. 
Thus $M \cong C_q$, and since $M$ is elementary abelian, it follows that $d = q$ is prime. 

It remains to show that $G_v^{\,\Gamma(v)}\cong C_q$ or $D_q$. 

Let $C$ be the largest cyclic subgroup of $H$, which has index $1$ or $2$ in $H$.
In fact $C = H$ in case (1), or the unique subgroup $\langle x\rangle$ of order $n$ 
and index $2$, in case (2). 
In each case, $C$ is a characteristic subgroup of $H$, and hence normal in $G_v$, 
and it follows that $C^{\,\Gamma(v)}$ is a normal subgroup of $G_v^{\,\Gamma(v)}$, 
of index at most $2$ in $H^{\,\Gamma(v)}$. 
Then because $G_v^{\,\Gamma(v)}$ is quasiprimitive,  $C^{\,\Gamma(v)}$ is transitive 
and hence regular. But now $C^{\,\Gamma(v)}$ must be a minimal normal subgroup
of $G_v^{\,\Gamma(v)}$, and in particular,  $C^{\,\Gamma(v)} = \soc(G_v^{\,\Gamma(v)}) \cong C_q$.
  
Next, since $C$ is a cyclic subgroup of $\SL(2,p)$ with order $|C|$ dividing $|H|$ and 
hence coprime to $p$, we find that $C$ is generated by some diagonal matrix of the form
  $$
 X= \left[
  \begin{matrix}
   \beta & 0 \\ 0 & \beta^{-1}
  \end{matrix}
  \right].
  $$
A direct computation shows that any matrix $A\in \GL(2,p)$ that conjugates $X$ to 
some power of $X$ is either diagonal or of the form
  $$
  \left[
  \begin{matrix}
   0 & b \\ c & 0
  \end{matrix}
  \right].
  $$ 
In both cases, $A^2$ is a diagonal matrix that commutes with $X$. 
%
Then since  $C$ is normal in $G_v$, it follows that the square of every 
element of $G_v$ centralises $C$, and hence the square of every element 
of $G_v^{\,\Gamma(v)}$ centralizes $C^{\,\Gamma(v)}$. 
But $G_v^{\,\Gamma(v)}$ is permutation isomorphic to a subgroup of $\AGL(1,q)$, 
with $C^{\,\Gamma(v)} = \soc(G_v^{\,\Gamma(v)})\cong C_q$ as a regular 
normal  subgroup, and so it follows that $G_v^{\,\Gamma(v)}$ is isomorphic to $C_q$ 
or to the dihedral group $D_q$, as required.
\end{proof}
\smallskip

\begin{lemma}
\label{lem:cases3&4}
Cases {\em (3)} and {\em (4)} are impossible. 
\end{lemma}

\begin{proof}
In these cases, $H\cong \hat{S}_4$ or $\SL(2,3)$, 
so again $H$ is soluble, and $G_v$ and $G_v^{\,\Gamma(v)}$ are soluble too, 
and then by Lemma~\ref{lem:bur}, the socle $M = \soc(G_v^{\,\Gamma(v)})$ is elementary 
abelian of order $d$, and is the only minimal normal subgroup of $G_v$. 
In particular, $H^{\,\Gamma(v)}$ contains $M$, of order $d \ge 5$. 
On the other hand, $H^{\,\Gamma(v)}$ is isomorphic to a quotient of $A_4$ or $S_4$, 
and so is isomorphic to one of $S_4$, $A_4$, $S_3$, $C_3$ and $C_2$, 
but none of these groups has an elementary abelian subgroup 
of order at least $5$, a contradiction.  
 \end{proof}

This leaves us with just case (5) to check, with $H\cong \SL(2,5)$.
\smallskip

\begin{lemma}
\label{lem:case5}
Case {\em (5)} is impossible. 
\end{lemma}

\begin{proof}
In this case $p \equiv \pm 1$ mod $5$, and $H\cong \SL(2,5)$,  
so $H/T \cong PSL(2,5) \cong A_5$. 
Since $G_v^{\,\Gamma(v)}\cong G_v/G_v^{[1]}$ has no normal subgroup of order $2$,
we know that $T = \langle -I_2 \rangle$ is a subgroup of $G_v^{[1]}$, 
and hence also of the arc-stabiliser $G_{uv}$. 
Next, $G_{uv}$ is a subgroup of $G_v$, so normalises $H$, and is a subgroup of $\GL(2,p)$. 
By Lemma~\ref{lem:charT} with $N = G_{uv}$, we conclude that  $T$ is characteristic in $G_{uv}$. 
Then since $G_{uv}$ is normal (of index $2$) in the edge-stabiliser $G_{\{u,v\}}$, 
it follows that $T$ is normal in $G_{\{u,v\}}$, and hence also normal in 
the group $G^* = \langle G_v, G_{\{u,v\}} \rangle$. 
But $\Gamma$ is connected and $G$ is transitive on the arcs of $\Gamma$, 
so $ \langle G_v, G_{\{u,v\}} \rangle = G$, and therefore $T$ is normal in $G$. 
On the other hand,  $T$ is contained in $H$ and hence in $G_v$, 
and so cannot be normal in $G$, for otherwise it would fix every vertex of $\Gamma$. 
This contradiction completes the proof. 
\end{proof}

 \section{Proof of Theorem~\ref{thm:f2}}
 \label{sec:proof-thmf2}

Theorem \ref{thm:f2} is an almost immediate consequence of Theorem \ref{thm:fg}. 

\smallskip 
For suppose $d \ge 4$, and $\Gamma$ is a connected $d$-valent $2$-arc-transitive 
graph of order $kp$ or $kp^2$ for some prime $p$, and let $G$ be a $2$-arc-transitive 
group of automorphisms of $\Gamma$.
Then $L = G_v^{\,\Gamma(v)}$ is $2$-transitive and therefore graph-restrictive, 
so Theorem \ref{thm:fg} applies, and we find that if $p \ge kc(L)$ then 
$G_v^{\,\Gamma(v)} \cong C_d$ or $D_d$. 
But the latter are not $2$-transitive, so this implies $p <  kc(L)$, 
and hence there are only finitely many possibilities for $p$, 
and so finitely many possibilities for $\Gamma$. 

Similarly, if $d = 3$ and $\Gamma$ has order $kp$, then only case (i) of part (d) 
of Theorem \ref{thm:fg} is possible, but in that case $G_v^{\,\Gamma(v)} \cong C_d$ 
when $p \ge kc(L)$, and again we find $p <  kc(L)$. 

\medskip 
We will consider the $3$-valent case further in Section \ref{sec:3valentcase}.

\section{Constructions}
\label{sec:constructions}

In this section, we consider what happens in case (i) of part (d) of Theorem~\ref{thm:fg}, 
in more detail. 
Before doing this, we describe a generic construction of arc-transitive graphs, 
which is often attributed to Subidussi.

\smallskip
Given a group $G$, a core-free subgroup $H$ of $G$, 
and an element $a\in G\setminus H$ such that $a^2\in H$, 
we may construct a graph $\Gamma(G,H,a)$ with vertex-set 
$(G\!:\!H) =\{Hg : g \in G\}$, and with two (right) cosets $Hx$ and $Hy$ 
adjacent if and only if $xy^{-1} \in HaH$.

Note that right multiplication of cosets by elements of $G$ gives rise to an arc-transitive 
and faithful action of $G$ on $\Gamma(G,H,a)$.  
The stabiliser of the vertex $H$ in the group $G$ is $H$, and multiplication 
by the element $a$ swaps the vertex $H$ with its neighbour $Ha$, 
and the stabiliser in $G$ of the arc $(H,Ha)$ is the subgroup $H\cap H^a$. 
Thus $\Gamma(G,H,a)$ is $G$-arc-transitive, 
with order $|G:H|$ and valency $|H:H\cap H^a|$.

The importance of this construction is reflected in the fact that every 
$G$-arc-transitive graph $\Lambda$ is isomorphic to $\Gamma(G,G_v,a)$, 
where $v$ is an arbitrary vertex of $\Lambda$ and $a$ is an arbitrary
element of $G$ swapping $v$ with a neighbour of $v$.

Also we note that two such graphs $\Gamma(G_1,H_1,a_1)$ and $\Gamma(G_2,H_2,a_2)$ 
are isomorphic whenever there is a group isomorphism 
$\varphi \colon G_1 \to G_2$ mapping $H_1$ to $H_2$ and $a_1$ to $a_2$.
On the other hand, is it sometimes possible for two such graphs to be isomorphic 
in other cases --- even when the groups $G_1$ and $G_2$ are not isomorphic --- 
since it can happen that a graph admits more than one arc-transitive group action. 

\smallskip
We can use the above construction to produce symmetric covers of a given 
arc-transitive graph, as in the proof of the following.  

\begin{theorem}
\label{thm:zeta} 
Let $d$ be an odd prime, let $k$ be an integer with $k\ge 3$, 
and let $p$ be any prime such that $p$ does not divide $k$, and $p \equiv 1$ mod $2d$. 
Also suppose that $\barG$ is a group of order $kd$
such that $\barG/[\barG,\barG] \cong C_{2d}$, and 
 that $\barGa$ is a $d$-valent $\barG$-arc-regular graph of order $k$. 
Then there exists at least one and at most $d-1$ pairwise non-isomorphic $\ZZ_p$-regular 
covering graphs $\Gamma$ for $\barGa$ 
such that the group $\barG$ lifts along the corresponding covering projection $\Gamma \to \barGa$.
Moreover, $\barG$ is generated by elements $h$ and $a$ such that $h^d = a^2 = 1$, 
with $h$ generating the stabiliser of some vertex ${\bar v}$, 
and $a$ interchanging $\bar v$ with one of its neighbours, and then each of these covering 
graphs  is isomorphic to a coset graph $\Gamma(G,\langle zh \rangle, a)$, where 
$G$ is a semi-direct product $C_p \rtimes \barG$, and $z$ is a generator for the 
normal subgroup $C_p$, with $z^a=z^{-1}$ and $z^h=z^\zeta$ 
for some primitive $d$-th root $\zeta$ of $1$ mod $p$. 
\end{theorem}

\begin{proof}
First, let $\Lambda$ be any connected $G$-arc-regular graph of order $n$ and 
valency $d$, with vertex-stabiliser $G_v$ ($\cong C_d$). 
Then $G$ has order $dn$, and is generated by an element of order $d$ stabilising a vertex and 
inducing a $d$-cycle on its neighbours, and an element of order $2$ that interchanges 
the vertex with one of its neighbours. Thus $G$ is a homomorphic image of 
the free product $\U = C_2 * C_d = \langle\, x, y \ | \  x^2 = y^d = 1 \,\rangle$,
which is a universal group for such arc-regular actions (for valency $d$).
Moreover, if $\eta\!: \mathcal U \to G$ is the corresponding epimorphism 
then $\Lambda$ is isomorphic to the graph $\Gamma(G,\langle y^\eta \rangle, x^\eta)$, 
and in particular, the order of $\Lambda$ is 
$| G \!:\! \langle y^\eta \rangle | = | \U \!:\! \langle N,y \rangle | 
=  | \U \!:\! N |/d$. 

Now take the given  $d$-valent $\barG$-arc-regular graph $\barGa$ of order $k$, 
with generators $h$ and $a$ for $\barGa$ being the images of $y$ and $x$, 
and suppose $\p \colon \Gamma \to \barGa$ is a $\ZZ_p$-regular covering projection for $\barGa$, 
such that $\barG$ lifts along $\p$.  If $G$ is the lift of $\barG$ along $\p$, 
then $G \cong \U/N$ and $\barG \cong \U/K$ for normal subgroups $K$ and $N$ of $\U$, 
such that $K$ has index $kd$ in $\U$, and $N$ is contained in $K$, 
with quotient $K/N$ being cyclic of order $p$. 
To make things easier, we might as well take $G = \U/N$ and $\barG = \U/K$. 

Since $|G| = kpd$, and $p$ is coprime to both $k$ and $d$, 
we find that $P = K/N$ is a cyclic normal Sylow $p$-subgroup of $G$, 
and so by Lemma~\ref{lem:GvGL}, we have $C_G(P) = [G,G] = J \times P$
for some normal subgroup $J$ of $G$.
Also by the Schur-Zassenhaus theorem, we know that $P$ has a complement 
in $G$, and because $G/P \cong (\U/N)/(K/N) \cong \U/K = \barG$ 
(and more importantly, the relations satisfied in $G/P$ by the images of the generators 
of $\U$ are the same as those satisfied by the images in $\barG$), 
we can take this complement to be $\bar G = \langle h, a \rangle$. 
Thus $G$ is isomorphic to a semi-direct product $C_p \rtimes \barG$. 

Next, the generators $a$ and $h$ of $\bar G$ induce automorphisms of $P$,
of orders $2$ and $d$, respectively, since they do not lie in $[G,G] = C_G(P)$.
Moreover, since $-1$ is the only unit of order $2$ mod $p$, it follows that conjugation 
by $a$ inverts every element of $P$, while conjugation by $h$ is exponentation 
by some primitive $d$-th root $\zeta$ of $1$ mod $p$. 

The value of $\zeta$ completely determines the structure of $G$, while on the other 
hand, the covering graph $\Gamma$ is determined by the choice of images of $x$ and $y$ in $G$. 

We can take the images of $x$ and $y$ in $G$ as $z_{1}a$ and $z_{2}h$ where $z_1, z_2 \in P$.  
These have orders $2$ and $d$,  
since $(z_{1}a)^2 = z_{1}z_{1}^a = z_{1}z_{1}^{-1} = 1$, 
and $(z_{2}h)^d = z_{2}^{1+\zeta+\zeta^2+\dots+\zeta^{d-1}}h^d = 1$, 
because $(1-\zeta)(1+\zeta+\dots+\zeta^{d-1}) \equiv 1 - \zeta^d \equiv 0$ mod $p$  
but $1-\zeta \not\equiv 0$ mod $p$. 
Also if we conjugate $z_{1}a$ and $z_{2}h$ by any element $u$ of $P$, 
then we get 
$$u^{-1}(z_{1}a)u = u^{-1}z_{1}u^{-1}a = z_{1}u^{-2}a 
\quad \hbox{and} \quad 
u^{-1}(z_{2}h)u = u^{-1}z_{2}u^{\zeta^{-1}}h = z_{2}u^{-1+\zeta^{-1}}h.$$ 
Since $p$ is odd, we can choose $u$ so that $u^2 = z_1$, and thereby assume 
that the image of $x$ is $a$ itself, and then we can take $z = z_{2}u^{-1+\zeta^{-1}}$, 
so that the image of $y$ is $zh$.  
Note that $z$ is non-trivial, since $a$ and $h$ generate $\bar G$ (rather than $G$), 
and in particular, $z$ is a generator for $P$. 

It follows that the covering graph $\Gamma$ is completely determined by $\zeta$,  
and since there are $d-1$ choices for $\zeta$, there are at most $d-1$ possibilities 
for $\Gamma$, as required.  

Finally, we show that every choice of $\zeta$ gives rise to such a covering $\Gamma$. 
To do this, we simply take $G$ as the semi-direct product $C_p \rtimes_{\zeta} \bar G$ 
given by $\zeta$, and let $z$ be any generator of the normal subgroup $C_p$, 
so that $z^a=z^{-1}$ and $z^h=z^\zeta$.  
Then $a$ and $zh$ are elements of orders $2$ and $d$ in $G$, and so 
we can construct the coset graph $\Gamma(G,\langle zh \rangle,a)$ in the usual way. 
Clearly this has $\Gamma(\bar G,\langle h\rangle,a) \cong \bar\Gamma$ as a quotient, 
and so all we have to do is prove that $a$ and $zh$ generate $G$. 
But now 
$$
[a,zh] = ah^{-1}z^{-1}azh = ah^{-1}z^{-2}ah = az^{-2\zeta}h^{-1}ah 
= z^{2\zeta}ah^{-1}ah = z^{2\zeta}[a,h], 
$$
and $[a,h]$ centralises $z$ (by the definition of $C_p \rtimes_{\zeta} \bar G$), 
so if $[a,h]$ has order $m$ in $\bar G$, then $[a,zh]^m = (z^{2\zeta}[a,h])^m = z^{2\zeta m}$, 
which is non-trivial (since $m$ divides $\bar G| = kd$ and hence is coprime to $p$), 
and therefore $[a,zh]^m$ generates the the normal subgroup $C_p$. 

This completes the proof.
\end{proof}

We note that it is sometimes possible that different choices for $\zeta$ give 
isomorphic covering graphs.  
For example, suppose $\Gamma = \Gamma(G,\langle zh \rangle,a)$ is the covering 
graph given by a particular value of $\zeta$, and there exists an automorphism of $G$ 
that inverts each of the generators $a$ and $zh$ for $G$.  
Then replacing $zh$ by $(zh)^{-1}$, we obtain a graph which is isomorphic to $\Gamma$. 
But $(zh)^{-1} = h^{-1}z^{-1} = z^{-\zeta}h^{-1}$, and conjugation by $h^{-1}$ is 
exponentiation by $\zeta^{-1}$, so it follows that $\zeta^{-1}$ gives the same graph $\Gamma$.

\section{The 3-valent (cubic) case}
\label{sec:3valentcase}

In the special case where the valency $d$ is $3$, we know from 
Theorem \ref{thm:f2} that for any fixed positive integer $k$, 
there exist only finitely many connected $2$-arc-transitive cubic 
graphs of order $kp$ where $p$ is prime.
But Theorems \ref{thm:fg} and \ref{thm:zeta} gives us more detailed information. 

Let $\Gamma$ be a connected symmetric cubic graph of order $kp$ where 
$k$ is a given even positive integer, and $p$ is a variable prime such that $p \ge 48k$,  
and  let $G$ be any arc-transitive group of automorphisms of $\Gamma$. 
Then by part (d)(i) of Theorem~\ref{thm:fg} with $L = C_3$ or $S_3$,  we know 
that $G$ has a cyclic normal Sylow $p$-subgroup $P$ which acts semiregularly 
on $\Gamma$, and that $\Gamma$ has a quotient $\bar{\Gamma}$ on 
which $\bar{G} = G/P$ acts arc-transitively.
Also the stabiliser $G_v$ is isomorphic to a subgroup of $\Aut(P)$ 
and is therefore cyclic.  
In particular, $G$ acts regularly on the arcs of $\Gamma$, 
with $G_v$ inducing $C_3$ on $\Gamma(v)$.  
Moreover, $\bar G$ has order $3k$,  
and $\,\bar{G}/[\bar{G},\bar{G}] \cong C_{6}$, and $p \equiv 1 $ mod $6$. 

In this case, $|G| = 3|V(\Gamma)| = 6p$, and by Theorem \ref{thm:zeta}, 
$\bar G$ is generated by elements $h$ and $a$ such that $h^3 = a^2 = 1$, 
and then $G$ is isomorphic to the semi-direct product $C_p \rtimes_{\lambda} \bar G$, 
where $\lambda$ is one of the two non-trivial cube roots of $1$ in $\Z_p$, 
and $h$ and $a$ conjugate a generator $z$ of the normal subgroup $C_p$ 
to $z^\lambda$ and $z^{-1}$ respectively. 

Note that there are just two possibilities for $\lambda$, and each is the inverse 
(or square) of the other.   
Also these two choices for $\lambda$ give non-isomorphic graphs, 
unless there exists an automorphism of the group $\bar G$ that inverts each 
of $h$ and $a$, which happens if and only if $\bar\Gamma$ admits a $2$-arc-regular 
group of automorphisms; see \cite{ConLor} or \cite{DjoMil}. 
 
Note also that if there is no finite group ${\bar G}$ of order $3k$ generated 
by two elements of orders $2$ and $3$ and with commutator subgroup 
$[{\bar G},{\bar G}]$ of index $6$ in ${\bar G}$, then there can be no connected 
symmetric cubic graph of order $kp$ for any prime $p \ge 48k$.

\smallskip
We can now apply this information to small values of $k$.

\subsection{Symmetric cubic graphs of order $2p$} 

Here $k = 2$, and the quotient graph $\bar{\Gamma}$ considered above
is $K_2$. However, by adjusting the definition of the quotient given in Section~\ref{sec:intro}
in a way that allows multiple edges in the quotient (see \cite{elabcov} for details),
one can view the quotient as the cubic dipole, with two vertices and three edges joining them. 
Although this $\bar{\Gamma}$ is not a simple graph, we can still view every 
large connected symmetric cubic graph of order $2p$ where $p$ is prime 
as a cover of $\bar{\Gamma}$.
In particular, $\bar{\Gamma}$ admits a $2$-arc-regular group of 
automorphisms.

Hence for each such prime $p \ge 96$ with $p \equiv 1$ mod $6$,  
there is exactly one arc-regular connected cubic group of order $2p$, 
and its automorphism group is a semi-direct product of $C_p$ by $C_3$.   
Many of these graphs appear in the lists of symmetric cubic graphs 
given in \cite{ConDob} and \cite{newlist}, for orders up to $768$ and $10000$ respectively. 

On the other hand, if $p < 96$, then every such $\Gamma$ has order less than $192$ 
and so appears in \cite{ConDob}.  
Apart from $1$-arc-regular examples with small $p \equiv 1$ mod $6$, 
but $p \ne 7$, there are just four such $\Gamma$, 
namely the $2$-arc-regular graph $\F004 \cong K_4$ (with $p = 2$), 
the $3$-arc-regular graphs $\F006 \cong K_{3,3}$ ($p = 3$), 
the $3$-arc-regular Petersen graph $\F010$ ($p = 5$), 
and the $4$-arc-regular Heawood graph $\F014$ ($p = 7$).  

In summary, every connected symmetric cubic graph of order $2p$ for some prime $p$ 
is either a uniquely determined $1$-arc-regular cubic graph (with $p \equiv 1$ mod $6$ and $p \ne 7$), 
or one of the four small exceptions ($K_4$, $K_{3,3}$, the Petersen graph or the Heawood graph).  
These are the same graphs as the ones found by Cheng and Oxley in \cite{ChengOxley} 
using a different approach.

\subsection{Symmetric cubic graphs of order $4p$} 

Here $k = 4$, but the only group of order $12$ that can be generated by two elements 
of orders $2$ and $3$ is the alternating group $A_4$, and in this group, 
the commutator subgroup has index $3$, not $6$. 
Hence there are no such graphs with $p \ge 48k = 192$. 
For smaller $p$, all the graphs appear in the census \cite{ConDob}. 
Again there are just four possibilities, 
namely the $3$-cube $Q_3 \cong \F008$ (which is $2$-arc-regular, with $p = 2$),  
the dodecahedral graph $\F020A$ and the canonical double 
cover $\F020B$ of the Petersen graph (which are $2$- and $3$-arc-regular respectively, with $p = 5$), 
and the $3$-arc-regular Coxeter graph $\F028$ (with $p = 7$). 
This was shown also in \cite[Section 6]{FengKwak}, by other means.

\subsection{Symmetric cubic graphs of order $6p$} 

Here $k = 6$, and we have just one connected symmetric cubic 
graph of order $k$, namely $K_{3,3}$, which is $3$-arc-transitive. 
Hence we have an infinite family of $1$-arc-regular cubic graphs 
of order $6p$, one for each large prime $p \equiv 1$ mod $6$.  
The only other graphs that arise in this case have order at most $48k^2 = 1728$, 
and so appear in the extended census \cite{newlist}. 
Apart from $1$-arc-regular examples with small $p \equiv 1$ mod $6$, 
there are just three such $\Gamma$, 
namely the $3$-arc-regular Pappus graph $\F018$ (with $p = 3$), 
Tutte's $8$-cage $\F030$ (which is $5$-arc-regular, with $p = 5$), 
and the $4$-arc-regular Sextet graph $S(17) \cong \F102$ (with $p = 17$).  
This classification was achieved also in \cite[Section 5]{FengKwak}, by different means. 
(Incidentally, there is a typographic error in the introduction 
of \cite{FengKwak}, where a claim is made about an infinite family of cubic 
2-arc-regular graphs of order $6p$; the order should be $6p^2$ (not $6p$).)

\subsection{Symmetric cubic graphs of order $8p$} 

Here $k = 8$, and we have just one connected symmetric cubic 
graph of order $k$, namely the cube graph $Q_3$, which is $2$-arc-transitive. 
Hence we have an infinite family of $1$-arc-regular cubic graphs 
of order $8p$, one for each large prime $p \equiv 1$ mod $6$.  
The only other graphs that arise in this case have order at most $48k^2 = 3072$, 
and so all of them appear in the extended census \cite{newlist}. 
Apart from $1$-arc-regular examples with small $p \equiv 1$ mod $6$,  
there are just five such graphs, 
namely the $2$-arc-regular graphs $\F016$ ($p = 2$) and $\F024$ ($p = 3$), 
the $3$-arc-regular graph $\F040$ ($p = 5$), and the graphs 
$\F056B$ and $\F056C$ (which are $2$- and $3$-arc-regular respectively, with $p = 7$).  
This classification was achieved also in \cite{FengKwakWang}, by different means.

\subsection{Symmetric cubic graphs of order $10p$} 

Since there is no group of order $30$ generated by two elements of 
orders $2$ and $3$, there are only finitely many connected symmetric cubic 
graphs of order $10p$ for $p$ prime.  
Moreover, since $p < 48k = 480$ for these, all such graphs appear in the extended 
census \cite{newlist}; but in fact they all appear in \cite{ConDob}. 
Again there are just five possibilities, 
namely the dodecahedral graph $\F020A$ and the canonical double cover $\F020B$ 
of the Petersen graph (which are $2$- and $3$-arc-regular respectively, with $p = 2$), 
Tutte's 8-cage $\F030$ (which is $5$-arc-regular, with $p = 3$), 
the graph $\F050$ (which is a $2$-arc-regular Cayley graph 
for the group $C_5 \wr S_2$, with $p = 5$) and the $3$-arc-regular Coxeter-Frucht graph 
$\F110$ ($p = 11$).  This classification was achieved also in \cite{FengKwak10p},  
by other means.

\subsection{Symmetric cubic graphs of order $12p$} 

Since there is no group of order $36$ generated by two elements of 
orders $2$ and $3$, there are only finitely many connected symmetric cubic 
graphs of order $12p$ for $p$ prime.  
Moreover, since $p < 48k = 576$ for these, all such graphs appear in the extended census \cite{newlist}.  
It is easily checked that there are just four possibilities,  
namely the $2$-arc-regular graphs $\F024$ ($p = 2$), $\F060$ ($p = 5$) 
and $\F084$ ($p = 7$), and the $4$-arc-regular graph $\F204$ ($p = 17$), 
all of which appear in \cite{ConDob}. 
This classification appears to be new.

\subsection{Symmetric cubic graphs of order $14p$} 

Here $k = 14$, and there is just one  possibility for $\bar\Gamma$, namely the 
Heawood graph $F014$, but this is $4$-arc-transitive and admits two arc-regular 
actions, one via the group $C_7 \rtimes_2 C_6$ and another via $C_7 \rtimes_4 C_6$, 
but admits no $2$-arc-regular action. It follows that for every large prime $p \equiv 1$ mod $6$, 
there are two non-isomorphic arc-regular connected cubic graphs of order $14p$. 
All other graphs that arise in this case have order at most $48k^2 = 9408$, 
and so appear in the extended census \cite{newlist}. 
Apart from other pairs of $1$-arc-regular examples with $p \equiv 1$ mod $6$ for $p > 7$,  
there are just six such graphs, and all of them appear in \cite{ConDob}. 
The exceptions are the $3$-arc-regular Coxeter graph $\F028$ (with $p = 2$), 
the $1$-arc-regular graphs $\F042$ (with $p = 3$) and $\F098B$ (with $p = 7$), 
the $2$-arc-regular graph $\F098B$ (also with $p = 7$), 
and the graphs $\F182C$ and $\F182D$ (which are $2$- and $3$-arc-regular, with $p = 13$).  
This classification seems to be new, as well.  

Note that all arc-transitive abelian regular covers of the graphs $K_{3,3}$ and $Q_3$ and the 
Heawood graph (encountered in the cases $k = 6, 8$ and $14$ above) are described 
in the papers \cite{ConMa,ConMa2}.

\subsection{Symmetric cubic graphs of order $kp$ for larger $k$} 

For slightly larger values of $k$, some more sophisticated arguments can be 
used to deduce the existence of a cyclic normal Sylow $p$-subgroup 
for many values of $p$ less than $48k$.  Even without going into those, 
we know the following:  

When $k = 16$, the graph $\Gamma$ is either a uniquely determined $1$-arc-regular 
cubic graph (with $p \equiv 1$ mod $6$) or one of a small finite list of exceptions, 
which includes the $2$-arc-regular graphs $\F032$, $\F048$ and $\F112B$, 
and the $3$-arc-regular graphs $\F080$ and $\F112B$.

When $k = 18$, the graph $\Gamma$ is either a uniquely determined $1$-arc-regular 
cubic graph (with $p \equiv 1$ mod $6$), or one of a small finite list of exceptions, 
which includes the $2$-arc-regular graph $\F054$, 
and the $5$-arc-regular graphs $\F090$ and $\F234B$. 

When $k = 20$, the graph $\Gamma$ is one of only a small finite number of possibilities,  
which include the $2$-arc-regular graph $\F060$, $\F220A$ and $\F220B$, 
the $3$-arc-regular graphs $\F040$ and $\F220C$, and 
the $4$-arc-regular graph $\F620$. 

Moreover, it is not difficult to obtain the theorem below.

\begin{theorem}
There are infinitely many $($even$)$ values of $k$ for which there are 
only finitely many connected symmetric cubic graphs of order $kp$. 
In particular, this is true for all $k$ of the form $2\ell$ where $\ell$ is 
a prime congruent to $5$ mod $6$.  
\end{theorem}

\begin{proof}
Let $k = 2\ell$ where $\ell$ is as given. 
By Theorem~\ref{thm:fg}, in order for there to exist infinitely many such graphs, 
there must be a finite group $\bar G$ of order $3k$ 
generated by two elements $a$ and $h$ of orders $2$ and $3$, 
with commutator subgroup $\bar G' = [\bar G,\bar G]$ of order $\ell$ and index $6$.  
Then since $\ell$ is prime, $\bar G'$ is cyclic. 
Also the generator $h$ of order $3$ for $\bar G$ must centralise $\bar G'$  
(since $\ell \not\equiv 1$ mod $3$), and it follows that $\bar G$ has a cyclic 
normal subgroup of order $3\ell$ and index $2$.  
This contains subgroups of orders $3$ and $\ell$ that are 
characteristic in $\bar G'$ and hence normal in $\bar G$, 
and then factoring out the characteristic subgroup $\bar H$ of order $3$ 
gives a dihedral quotient $\bar G/\bar H$ of order $2\ell$.  
But on the other hand, since $\bar H$ contains the generator $h$ of order $3$, 
this quotient $\bar G/\bar H$ is generated by the image of the involuntary element $a$,
which is clearly impossible.  
\end{proof}

Finally, we note that this argument works also for other values of $k$ for which there 
exists no finite group of order $3k$ generated by two elements of orders $2$ and $3$, 
with commutator subgroup of index $6$.   
For small $k$, checking for the existence of groups with the required 
properties is an easy exercise using {\sc Magma} \cite{magma}. 

Hence, for example, there are only finitely many connected symmetric cubic 
graphs of order $kp$ for $p$ prime when 
$k = $ 4, 10, 12, 20, 22, 28, 30, 34, 36, 40, 44, 46, 52, 58, 60, 66, 68, 70, 76, 80,
82, 84, 88, 90, 92, 94 or 100; 
and on the other hand, there is an infinite family of such graphs whenever 
$k = $ 2, 6, 8, 14, 16, 18, 24, 26, 32, 38, 42, 48, 50, 54, 56, 62, 64, 72, 74, 78, 86, 96 or 98.

\bigskip
\bigskip
\centerline{\sc Acknowledgements} 

The authors acknowledge the use of {\sc Magma} \cite{magma} in testing 
various matters considered in this paper. 

The first author was generously supported by a James Cook Fellowship from the 
Royal Society of New Zealand, and a grant from the N.Z. Marsden Fund.
\bigskip


\bigskip\bigskip

\end{document}